\pdfoutput=1
\documentclass[11pt,letterpaper]{amsart}
\usepackage{xcolor}
\usepackage{graphicx}
\usepackage{amsmath,amsfonts,amssymb,graphics,amsthm}
\usepackage{hyperref}
\usepackage{comment}
\usepackage{tabularx}
\usepackage[protrusion=true,expansion=true]{microtype}
\usepackage{enumerate}
\usepackage{bbm}
\usepackage{mathrsfs}
\usepackage[margin=1in]{geometry}
\usepackage[normalem]{ulem}

\newcommand{\aexp}{a}
\newcommand{\bexp}{b}

\hypersetup{
    colorlinks=false,
    linktocpage,
    }

\numberwithin{equation}{section}

\newtheorem{theorem}{Theorem}[section]

\newtheorem{lemma}[theorem]{Lemma}
\newtheorem{proposition}[theorem]{Proposition}

\newtheorem{observation}[theorem]{Observation}  

\newtheorem{remark}[theorem]{Remark}
\newtheorem{definition}[theorem]{Definition}

\theoremstyle{remark}

\newcommand{\C}{\mathbbm{C}}

\newcommand{\E}{\mathbbm{E}}
\newcommand{\N}{\mathbbm{N}}

\newcommand{\R}{\mathbbm{R}}

\newcommand{\eps}{\varepsilon}

\newcommand{\sph}{\mathrm{sph}}
\newcommand{\disk}{\mathrm{disk}}

\newcommand{\cMtwo}{\mathcal{M}_{2}^\mathrm{disk}}

\newcommand{\LF}{\mathrm{LF}}
\newcommand{\QD}{\mathrm{QD}}
\newcommand{\QS}{\mathrm{QS}}
\newcommand{\MD}{\mathrm{MD}}
\newcommand{\MS}{\mathrm{MS}}

\newcommand{\conf}{\operatorname{conf}}

\newcommand{\haar}{\mathbf{m}}
\newcommand{\lp}{\mathrm{loop}}
\newcommand{\cont}{\mathrm{Cont}}

\let\Re\undefined
\DeclareMathOperator{\Re}{Re}

\DeclareMathOperator{\SLE}{SLE}
\DeclareMathOperator{\CLE}{CLE}

\def\cS{\mathcal{S}}

\def\cM{\mathcal{M}}

\def\cF{\mathcal{F}}
\def\cE{\mathcal{E}}
\def\cD{\mathcal{D}}
\def\cC{\mathcal{C}}

\def\alb#1\ale{\begin{align*}#1\end{align*}}
\def\allb#1\alle{\begin{align}#1\end{align}}

\newcommand{\aryb}{\begin{eqnarray*}}
\newcommand{\arye}{\end{eqnarray*}}
\def\alb#1\ale{\begin{align*}#1\end{align*}}
\newcommand{\eqb}{\begin{equation}}
\newcommand{\eqe}{\end{equation}}
\newcommand{\eqbn}{\begin{equation*}}
\newcommand{\eqen}{\end{equation*}}

\newcommand{\BB}{\mathbbm}

\newcommand{\op}{\operatorname}

\newcommand{\frk}{\mathfrak}

\newcommand{\rta}{\rightarrow}

\newcommand{\wt}{\widetilde}
\newcommand{\wh}{\widehat}

\let\originalleft\left
\let\originalright\right
\renewcommand{\left}{\mathopen{}\mathclose\bgroup\originalleft}
\renewcommand{\right}{\aftergroup\egroup\originalright}

\DeclareMathAlphabet{\mathpzc}{OT1}{pzc}{m}{it}

\begin{document}

\title[The SLE loop via conformal welding]{The SLE loop via conformal welding of quantum disks}
\author{Morris Ang, Nina Holden, and Xin Sun} 

\maketitle

\begin{abstract}
We prove that the SLE$_\kappa$ loop measure arises naturally from the conformal welding of two $\gamma$-Liouville quantum gravity (LQG) disks for $\gamma^2=\kappa\in(0,4)$. The proof relies on our companion work on conformal welding of LQG disks and uses as an essential tool the concept of \emph{uniform embedding} of LQG surfaces. Combining our result with work of Gwynne and Miller, we get that random quadrangulations decorated by a self-avoiding polygon converge in the scaling limit to the LQG sphere decorated by the SLE$_{8/3}$ loop. Our result is also a key input to recent work of the first and third coauthors on the integrability of the conformal loop ensemble. Finally, our result can be viewed as the random counterpart of an action functional identity due to Viklund and Wang.
\end{abstract}

\section{Introduction}
% !TeX spellcheck = en_US

The Schramm-Loewner evolution (SLE)  and Liouville quantum gravity (LQG) are central objects in random conformal  geometry. It was shown by Sheffield~\cite{shef-zipper} and Duplantier-Miller-Sheffield~\cite{wedges} that SLE curves arise as the interfaces of LQG surfaces under conformal welding. This phenomenon 
 is a cornerstone of the mating-of-trees framework~\cite{wedges} for the SLE/LQG coupling and is a fundamental input to the link between LQG and the scaling limits of random planar maps. See e.g.~\cite{lawler-book,ghs-mating-survey,gwynne-ams-survey, berestycki-lqg-notes,shef-ICM}  for an introduction to SLE, LQG, and their interactions.

Conformal welding results in~\cite{shef-zipper,wedges} mainly focus on infinite-volume LQG surfaces. Recently in~\cite{ahs-disk-welding}, we showed that the conformal welding of finite-volume LQG surfaces called two-pointed quantum disks can give rise to some canonical  variants of SLE curves with two marked points. In this paper, we show in Theorem~\ref{thm-loop} that when  conformally welding two quantum disks without marked points, the interface is another  canonical variant of SLE called the \emph{SLE loop}. Moreover, the resulting LQG surface is the so-called \emph{quantum sphere} (without marked point), which describes the scaling limit of classical random planar map models with spherical topology. For example, in the pure gravity case it corresponds to the Brownian map~\cite{legall-uniqueness,miermont-brownian-map,lqg-tbm1}.

As reviewed in Section~\ref{subsec:intro-loop}, the SLE loop is an important one-parameter family of conformally invariant random Jordan curves whose existence was conjectured by  Kontsevitch and Suhov \cite{ks-loop} and  established by Zhan~\cite{zhan-loop-measures}. In particular, the $\SLE_{8/3}$ loop introduced by Werner~\cite{werner-loops} describes the conjectural scaling limit of self-avoiding polygons on planar lattices. Our conformal welding result Theorem~\ref{thm-loop} combined with earlier work of Gwynne and Miller~\cite{gwynne-miller-gluing,gwynne-miller-saw} yields that uniform quadrangulations decorated by a self-avoiding loop converge to the Brownian map decorated by the SLE$_{8/3}$ loop; see Theorem~\ref{thm:saw-loop}. For $\kappa\in (8/3,4]$, the SLE$_\kappa$ loop is closely related to  the conformal loop ensemble (CLE) considered in~\cite{shef-cle,shef-werner-cle}.

We will state Theorems~\ref{thm-loop} and~\ref{thm:saw-loop} in Section~\ref{subsec:intro-loop} and~\ref{subsec:rmp}, respectively, modulo some background  material supplied in Section \ref{sec-prelim}.
Then we prove Theorems \ref{thm-loop} and~\ref{thm:saw-loop} in  Sections~\ref{sec-loop} and~\ref{subsec:rpm}, respectively.
Our proof of Theorem~\ref{thm-loop} builds on the conformal welding result in~\cite{ahs-disk-welding} and the Liouville field description of the quantum surfaces  in~\cite{AHS-SLE-integrability}, which will be recalled in Section~\ref{sec-prelim} as well.

\subsection{The SLE loop via conformal welding}\label{subsec:intro-loop}
Kontsevitch and Suhov \cite{ks-loop}, inspired by Malliavin \cite{Malliavin}, formulated a natural conformal restriction covariance property for loop measures, and they conjectured that there is a unique (up to multiplicative factor) one-parameter family of loop measures satisfying this property. We call loop measures satisfying this property a \emph{Malliavin-Kontsevitch-Suhov (MKS) loop measure} and we index the measures by $\kappa\in(0,4]$.

Werner \cite{werner-loops} proved the existence and uniqueness of the MKS loop measure for $\kappa=8/3$ and constructed the loop measure via the boundaries of  Brownian loops.  Kemppainen and Werner \cite{werner-sphere-cle} constructed an MKS loop  measure  for $\kappa\in(8/3,4]$  
using the density measure of a nested simple CLE. 
For $\kappa=2$, Benoist and Dub\'edat~\cite{benoist-sle2} proved that a measure constructed in  \cite{kk-curves-laplacian} is an MKS loop measure.  Finally, Zhan~\cite{zhan-loop-measures} constructed an MKS loop measure for all $\kappa\in(0,4]$ via SLE$_\kappa$ equipped with its natural parametrization and also extended the construction to $\kappa\in(4,8)$.
We denote Zhan's MKS loop measure by $\SLE_{\kappa}^{\op{loop}}$. See Section \ref{sec-zhan-loop-measure} for the precise definition. 
For $\kappa\in (8/3,4]$, the loop measures of Zhan and Kemppainen-Werner  agree; see~\cite[Section 2.3]{AS-CLE}.
We emphasize that $\SLE_{\kappa}^{\op{loop}}$ is an infinite measure for each $\kappa\in(0,4]$.

{For each $\gamma\in(0,2)$ there is a natural infinite measure on LQG surfaces with spherical topology called the \emph{unmarked quantum sphere}. We denote this measure by $\QS$ and refer to Section \ref{sec-prelim} for the precise definition of both this measure and of the other objects we introduce in this and the next paragraph. Let $\wh \C = \C \cup \{\infty\}$ be the Riemann sphere. Suppose $h$ is a random field on $ \C$ such that  
the distribution of $(\wh \C,h)$ viewed as a quantum surface is $\QS$.
Let  $\eta$  be a sample\footnote{We will use the language of probability in the setting of non-probability measures --- see Section~\ref{sec-prelim} for precise definitions.} of $\SLE^{\mathrm{loop}}_\kappa$ independent of $h$ for $\kappa=\gamma^2\in(0,4]$. It is known that $\SLE^{\mathrm{loop}}_\kappa$ is invariant under M\"{o}bius transforms \cite[Theorem 4.2]{zhan-loop-measures}. Therefore, although the distribution of $h$ is not uniquely specified,
as a curve-decorated quantum surface, the distribution of $(\wh \C,h,\eta)$ is uniquely defined. 
We denote this loop-decorated quantum surface by $\QS\otimes \SLE_{\kappa}^{\op{loop}}$ and call it the \emph{MKS-loop-decorated quantum sphere} with parameter $\gamma$; see Section~\ref{sec-loop} for a precise definition.

For each $\gamma\in(0,2)$ there is also a natural infinite measure on LQG surfaces with disk topology called the \emph{unmarked quantum disk}. We denote this measure by $\QD$. Let $\QD(\ell)$ be the disintegration of $\QD$ over its boundary length, namely  $\QD=\int_0^\infty\QD(\ell) \, d \ell$. For $\ell > 0$, let $(\cD_1, \cD_2)$ be a sample from $\QD(\ell)\times\QD(\ell)$, so $\cD_1$ and $\cD_2$ have boundary length $\ell$.
 Let $\mathrm{Weld}(\cD_1,\cD_2)$ be the  curve-decorated quantum surface obtained by conformally welding $\cD_1$ and $\cD_2$ along their boundaries such that a uniformly sampled point on the boundary of $\cD_1$ is identified with a uniformly sampled point on the boundary of $\cD_2$.
We denote the  distribution of $\mathrm{Weld}(\cD_1,\cD_2)$ by $\mathrm{Weld}(\QD(\ell),\QD(\ell))$ and define $\mathrm{Weld}(\QD,\QD)=\int_0^\infty \ell\cdot  \mathrm{Weld}(\QD(\ell),\QD(\ell)) \, d\ell$; 
see Section~\ref{subsec:welding} for more details.
\begin{theorem}\label{thm-loop}
For $\gamma\in (0,2)$ and $\kappa = \gamma^2$, we have $\QS\otimes \SLE_{\kappa}^{\op{loop}} = C\mathrm{Weld}(\QD,\QD)$ for some  constant $C \in (0, \infty)$. 	
\end{theorem}

The proof of Theorem~\ref{thm-loop} builds on a  welding result from~\cite{ahs-disk-welding} which says that the conformal welding of two quantum disks, each marked with two points sampled independently from the LQG boundary measure, gives a quantum sphere with two special singularities decorated with a so-called two-sided whole plane $\SLE_\kappa$. 
Given this result and the construction of Zhan's MKS loop measure, Theorem \ref{thm-loop} seems plausible. To explain the factor of $\ell$ in the definition of $\mathrm{Weld}(\QD, \QD)$, we appeal to the following intuition from the discrete: if we have two polygons with $\ell$ edges, there are $\ell$ different ways to glue them into a sphere with a self-avoiding loop. The rigorous proof of Theorem~\ref{thm-loop} relies on the idea of uniform embedding introduced in \cite{AHS-SLE-integrability} and explained in Section~\ref{subsec-haar}. The uniform embedding of a quantum surface is a particular random embedding where the fields can be described via Liouville conformal field theory (LCFT) as considered in~\cite{dkrv-lqg-sphere,hrv-disk}. Moreover, it is especially convenient  to work with when  adding or removing marked points on the surface.

Based on the integrability of LQG from mating-of-trees~\cite{wedges} and LCFT (see e.g.~\cite{krv-dozz,ARS-FZZ,rz-boundary}),  it was shown in~\cite{AHS-SLE-integrability} that conformal welding  can be applied to establishing integrability results for the SLE interfaces involved. Theorem~\ref{thm-loop} serves as the starting point of this application to the $\SLE$ loop and CLE. In~\cite{AS-CLE}, this approach  was used to obtain the 3-point correlation function for the nesting statistics and the electrical thickness  of simple $\CLE$. In~\cite{ARS-annulus}, it was used to compute the annulus partition function of the $\SLE_{8/3}$ loop.
In a forthcoming work of the first and third coauthors, it will be used to compute the renormalized probability that three given points are  close to the same 
CLE loop on the sphere.

By a limiting argument, Theorem~\ref{thm-loop} can be naturally extended to $\kappa=4$ using available conformal welding results for $\SLE_4$~\cite{hp-welding,mmq-welding} but since the conformal removability of $\SLE_4$ is not settled, the result would be less definite. Therefore  we do not pursue this extension.

Viklund and Wang \cite{viklund-wang-welding} proved a beautiful identity between the Loewner energy of a Jordan curve $\eta$ on a sphere and the difference between the Dirichlet energy of a function $\varphi$ on the sphere and the sum of the Dirichlet energies of $\varphi$ restricted to each component of $\BB S^2\setminus\eta$ after applying a uniformizing map.  These quantities naturally arise from the large deviation of  SLE and LCFT; see~\cite{wang-LDP-sle} and~\cite{lrv-semiclassical}.
In particular, the identity can be viewed as a relation between the large deviation rate functions for an SLE loop measure, the LCFT on the sphere,  and the LCFT on the disk. See \cite[Section 1.3]{viklund-wang-welding} for a discussion of this.
Theorem~\ref{thm-loop} can be viewed as a quantum  version of this identify.

\subsection{The scaling limit of random planar maps decorated by self-avoiding loop}\label{subsec:rmp}

It is believed that the MKS loop with $\kappa=8/3$, namely Werner's loop measure~\cite{werner-loops}, 
is the scaling limit of the critical self-avoiding loop on a regular planar lattice. We will argue that this result holds in an annealed sense in the setting of planar maps. Namely, the critical Boltzmann measure on quadrangulations decorated with a self-avoiding loop converges to the MKS-loop-decorated quantum sphere with $\gamma=\sqrt{8/3}$ as curve-decorated metric measure spaces. The measure is called critical since the weight assigned to a loop-decorated quadrangulation has been tuned precisely so that the number of vertices of the quadrangulation and the length of the loop have a power-law behavior.

To state this result we will first introduce some notation; see Section \ref{subsec:rpm} for more precise definitions. Let $\MS^{n}\otimes \op{SAW}^n$ denote the measure on pairs $(M,\eta)$ where $M$ is a quadrangulation, $\eta$  is a self-avoiding loop on $M$, and a pair $(M,\eta)$ has weight 
$n^{2.5}12^{-\#\cF(M)}54^{-\#\eta},$
where $\#\cF(M)$ is the number of faces of $M$ and $\#\eta$ is the number of edges of $\eta$. Note that the parameter $n$ does not correspond to any quantity in the quadrangulation beyond the weights we use to define MS$^n$ and as a scaling factor for distances and areas (see Section \ref{subsec:rpm} for the latter). We let $\QS\otimes \SLE_{8/3}^{\op{loop}}$ be as in Theorem~\ref{thm-loop}
with $\gamma^2=\kappa=8/3$. 
As we will explain in Section \ref{subsec:rpm}, samples from
$\MS^{n}\otimes \op{SAW}^n$ and $\QS\otimes \SLE_{8/3}^{\op{loop}}$ can be viewed
as loop-decorated metric measure spaces. In this setting the natural topology for weak convergence  is the Gromov-Hausdorff-Prokhorov-uniform (GHPU) topology; see Section \ref{sec:ghpu} for a precise definition. For a loop-decorated metric measure space and $c\in(0,1)$ we let $A(c)$ denote the event that the length of the loop is in $[c,c^{-1}]$. We use the notation $M|_{A(c)}$ to stand for the restriction of the measure $M$ to the event $A(c)$, and
use the symbol $\Rightarrow$ to indicate weak convergence of finite measures.
\begin{theorem}\label{thm:saw-loop}
	There exists $c_0>0$ such that for all $c\in(0,1)$, $$\MS^{n}\otimes \op{SAW}^n|_{A(c)} \Rightarrow c_0\cdot\QS\otimes \SLE_{8/3}^{\op{loop}}|_{A(c)} \quad \textrm{as }n\rta\infty $$ 
	with respect to the Gromov-Hausdorff-Prokhorov-uniform topology.
\end{theorem}
Here, we restrict to $A(c)$ to make all measures in Theorem~\ref{thm:saw-loop} finite. 
Gwynne and Miller proved the counterpart of the theorem in the setting of chordal self-avoiding paths on half-planar quadrangulations \cite{gwynne-miller-saw,gwynne-miller-gluing}, and results from their papers are key inputs to our proof. The other inputs are Theorem \ref{thm-loop} and an exact discrete counterpart (also observed in \cite{gwynne-miller-simple-quad,caraceni-curien-saw}) of Theorem \ref{thm-loop} given in Observation~\ref{obs:bijection}. 

It is a classical result that the quantum sphere for $\gamma=\sqrt{8/3}$ (also known as the Brownian map) arises as the scaling limit of uniformly sampled planar maps \cite{legall-uniqueness,miermont-brownian-map}. By contrast, our Theorem \ref{thm:saw-loop} gives the analogous result for a family of \emph{non-uniform} planar maps in the sense that two planar maps of the same size do not have the same probability of being sampled. Indeed, if $(M,\eta)$ is sampled from $\MS^{n}\otimes \op{SAW}^n$ then the marginal law of $M$ has been reweighted according to the (weighted) number of self-avoiding loops that the map admits. On the other hand, Theorem~\ref{thm:saw-loop} means that  this reweighting does not change the scaling limit of the planar map.

For concreteness Theorem \ref{thm:saw-loop} is stated and proved for quadrangulations, but we remark that the result also holds for random triangulations building on \cite{aasw-type2}.\footnote{Ewain Gwynne has confirmed in private communication that the techniques in his self-avoiding walk papers with Jason Miller also work for triangulations.} By universality we expect that Theorem \ref{thm:saw-loop} also extends to other families of planar maps decorated by a self-avoiding loop.

\medskip

\noindent\textbf{Acknowledgements.} 
We are in debt to Yilin Wang for her important insight on SLE loop.
In our opinion, her contribution to Theorem~\ref{thm-loop}  is as much as ours.
We are also grateful to  Steffen Rohde, Scott Sheffield, and  Dapeng Zhan for helpful discussions. 
We thank two anonymous referees for helpful comments on an
earlier version of this paper.
M.A. was supported by the Simons Foundation as a Junior Fellow at the Simons Society of Fellows, and partially supported by NSF grant DMS-1712862. 
N.H.\ was supported by Dr.\ Max R\"ossler, the Walter Haefner Foundation, and the ETH Z\"urich
Foundation, along with grant 175505 of the Swiss National Science Foundation.
X.S.\ was supported by  the Simons Foundation as a Junior Fellow at the Simons Society of Fellows,  and by the  NSF grant DMS-2027986 and the Career award 2046514.

\section{Preliminaries}\label{sec-prelim}

In this paper we use the language of probability theory in the setting of non-probability measures. If $M$ is a measure on a measurable space $(\Omega, \cF)$ and $X$ is an $\cF$-measurable function, we call the pushforward measure $M_X = X_* M$ the \emph{law} of $X$, and say that $X$ is \emph{sampled} from $M_X$. For a finite measure $M$, we denote its total mass by $|M|$, and write $M^\# = |M|^{-1} M$ for the probability measure proportional to $M$. We now provide background for the various objects relevant to Theorem~\ref{thm-loop}.

\subsection{Liouville quantum gravity}\label{subsec-lqg}

We introduce the \emph{Gaussian free field} (GFF) on various domains. Let $\cS= \R \times (0,\pi)$ be the infinite strip and let $m$ be the uniform probability measure on $\{0\} \times (0,\pi)$. The Dirichlet inner product is given by $\langle f,g\rangle_\nabla = (2\pi)^{-1} \int_\cS \nabla f \cdot \nabla g$. Consider the space of smooth functions $f$ on $\cS$ with $\langle f, f\rangle_\nabla < \infty$ and $\int_\cS f \, dm = 0$. Let $H(\cS)$ be its Hilbert space closure with respect to $\langle \cdot, \cdot \rangle_\nabla$. Choose an orthonormal basis $(f_i)$ of $H(\cS)$ and let $(\xi_i)$ be independent standard Gaussian random variables. Then the summation 
\[h_\cS := \sum_{i=1}^\infty \xi_i f_i \]
converges in the space of distributions, and we call $h_\cS$ a GFF on $\cS$ normalized so that $\int_\cS h_\cS \, dm= 0$ \cite[Section 4.1.4]{wedges}. 

Throughout this paper, we fix a choice of LQG parameter $\gamma \in (0,2)$.
Suppose $\phi = h_\cS+ g$ where $g$ is a (possibly random) function on $\cS \cup \partial \cS$ which is continuous at all but finitely many points. For $z \in \cS \cup \partial \cS$ let $\phi_\eps(z)$ be the average of $\phi$ on $\partial B_\eps(z) \cap \cS$, and define $\mu_\phi^\eps(d^2z) := \eps^{\gamma^2/2} e^{\gamma \phi_\eps(z)} \, d^2z$ where $d^2z$ is the Lebesgue measure on $\cS$. The \emph{quantum area measure} $\mu_\phi$ is defined as the almost sure weak limit $\lim_{\eps \to 0} \mu_\phi^\eps$ \cite{shef-kpz, sw-coord}. Similarly, we can define the \emph{quantum boundary length measure} $\nu_\phi:= \lim_{\eps \to 0} \eps^{\gamma^2/4} e^{\frac\gamma2 \phi_\eps(x)} \, dx$ where $dx$ is the Lebesgue measure on $\partial \cS$. 

Suppose $f: D \to \wt D$ is a conformal map between domains $D, \wt D$. For a distribution $\phi$ on $D$, define
\eqb \label{eq:coord}
f \bullet_\gamma \phi = \phi \circ f^{-1} + Q \log |(f^{-1})'|, \qquad \qquad Q = \frac\gamma2 + \frac2\gamma.
\eqe
Consider the set of pairs $(D, \phi)$ where $D \subset \C$ is open and $\phi$ is a distribution on $D$. A \emph{quantum surface} is an equivalence class of pairs $(D,\phi)$ where $(D,\phi) \sim_\gamma (\wt D, \wt h)$ if there is a conformal map $f: D \to \wt D$ such that $\wt \phi = f \bullet_\gamma \phi$, and an \emph{embedding} of the quantum surface is a choice of $(D,h)$ from the equivalence class. 
This definition is natural because the quantum area and boundary length measures are consistent across elements of an equivalence class: if $(\cS, \phi)\sim_\gamma (\cS, \wt \phi)$ and $f: \cS \to \cS$ satisfies $\wt\phi = \phi \circ f^{-1} + Q \log |(f^{-1})'|$, then $\mu_{\wt \phi} = f_* \mu_\phi$ and $\nu_{\wt \phi} = f_* \nu_{\phi}$ \cite{shef-kpz}.

More generally, a \emph{loop-decorated quantum surface with $m$ marked points} is an equivalence class of tuples $(D, \phi, z_1, \dots, z_m, \eta)$ with $z_1, \dots, z_m \in D \cup \partial D$, $\eta: \mathbb S^1_\ell \to D$ is continuous (i.e., $\eta$ is a loop on $D$), and $\mathbb S^1_\ell$ is a circle of length $\ell>0$. We say $(D, \phi, z_1, \dots, z_m, \eta)\sim_\gamma(\wt D, \wt \phi, \wt z_1, \dots, \wt z_m, \wt\eta)$ if $ \wt \phi = f\bullet_\gamma \phi$, $\wt z_i = f(z_i)$ for all $i$, and $\wt \eta(t) = f(\eta(t+r))$ for some $r \in [0,\ell)$ and all $t \in [0,\ell)$, where we represent $\mathbb S_\ell^1$ as the interval $[0,\ell]$ with endpoints identified. We view $\eta$ as a parametrized and oriented loop with no distinguished starting point. We can similarly define a quantum surface with just $m$ marked points (and no loop).

Now, we recall the radial-lateral decomposition of $h_\cS$. Let $H_1(\cS)$ (resp.\ $H_2(\cS)$) be the subspace of $H(\cS)$ comprising functions which are constant (resp.\ have average zero) on $\{t\} \times (0,\pi)$ for each $t \in \R$. This yields an orthogonal decomposition $H(\cS) = H_1(\cS) \oplus H_2(\cS)$. 

\begin{definition}\label{def-thick-disk}
	Let $W \geq \frac{\gamma^2}2$ and $\beta = Q + \frac\gamma2 - \frac W{\gamma^2}$. Let 
\[Y_t =
\left\{
\begin{array}{ll}
	B_{2t} - (Q -\beta)t  & \mbox{if } t \geq 0 \\
	\wt B_{-2t} +(Q-\beta) t & \mbox{if } t < 0
\end{array}
\right. , \]
where $(B_s)_{s \geq 0}$ is a standard Brownian motion  conditioned on $B_{2s} - (Q-\beta)s<0$ for all $s>0$, and $(\wt B_s)_{s \geq 0}$ is an independent copy of $(B_s)_{s \geq 0}$. Let $h^1(z) = Y_{\Re z}$ and let $h^2_\cS$ be the projection of an independent GFF $h_\cS$ to $H_2(\cS)$. Set $\wh h = h^1 + h^2_\cS$. Sample an independent real number $\mathbf c$ from the measure $[\frac\gamma2e^{(\beta - Q)c} \, dc]$ on $\R$, and let $\phi = \wh h + \mathbf c$. Let $\cM_2^\disk(W)$ be the infinite measure describing the law of $(\cS, \phi, -\infty, +\infty)/{\sim_\gamma}$. We call a sample from 	$\cM_2^\disk(W)$ a \emph{quantum disk with two insertions of weight $W$}. 
\end{definition}
Weight $W$ quantum disks have two marked boundary points.
The case $W=2$ is special since these two points are quantum typical in the following sense.

\begin{proposition}[{\cite[Proposition A.8]{wedges}}]\label{prop-QD-typical}
	Sample $(\cS, \phi, -\infty, +\infty)/{\sim_\gamma}$ from $\cM_2^\disk(2)$, then sample independent points $x_1, x_2 \in \partial \cS$ from the probability measure proportional to $\nu_\phi$. Then the law of $(\cS, \phi, x_1, x_2)/{\sim_\gamma}$ is $\cM_2^\disk(2)$. 
\end{proposition}
This motivates the following definition. 
\begin{definition}\label{def-QD}
	Sample $(\cS, \phi, -\infty, +\infty)/{\sim_\gamma}$ from the weighted measure $\nu_\phi(\partial \cS)^{-2} \cM_2^\disk(2)$. Then we call $(\cS, \phi)/{\sim_\gamma}$ a \emph{quantum disk}, and denote its law by $\QD$. 
\end{definition}
Here is a useful perspective on Proposition~\ref{prop-QD-typical} and Definition~\ref{def-QD}. Roughly speaking, given a sample $(\cS, \phi)/{\sim_\gamma}$ from $\QD$, if we ``sample two points from $\nu_\phi$'', then the resulting quantum surface with two marked points has law $\cM_2(2)$. Since $\nu_\phi$ is a non-probability measure with total mass $\nu_\phi(\partial \cS)$, the sampling operation should induce a weighting by $\nu_\phi(\partial \cS)^2$; this explains the factor $\nu_\phi(\partial \cS)^{-2}$ in Definition~\ref{def-QD}.

\begin{lemma}\label{lem:QD-len}
	The law of the quantum boundary length $\nu_\phi(\partial \cS)$ of a sample $(\cS, \phi)/{\sim_\gamma}$ from $\QD$ is $ 1_{\ell>0}C \ell^{-\frac4{\gamma^2}-2}\, d\ell$ for some $C>0$. 
\end{lemma}
\begin{proof}
	\cite[Lemma 3.3]{AHS-SLE-integrability} implies that the total boundary length of a sample from $\cM_2(2)$ has law $1_{\ell > 0} C \ell^{-\frac4{\gamma^2}}\, d\ell$. By Definition~\ref{def-QD}, weighting by $\ell^{-2}$ gives the corresponding result for $\QD$.
\end{proof}
Consequently, we can define a \emph{disintegration} $\{\QD(\ell)\}_{\ell>0}$ of $\QD$ on its quantum boundary length, i.e.\ $\QD= \int_0^\infty \QD(\ell) \, d\ell$ for measures $\QD(\ell)$ supported on the space of quantum surfaces with boundary length $\ell$. This only specifies $\QD(\ell)$ for a.e.\ $\ell$, but by continuity we can canonically define $\QD(\ell)$ for all $\ell$; see e.g.\ \cite[Section 4.5]{wedges} or \cite[Section 2.6]{ahs-disk-welding}.

Define the horizontal cylinder $\cC:= \R \times [0,2\pi]/{\sim}$ by identifying $(x,0)\sim (x, 2\pi)$ for all $x\in \R$. Let $m$ be the uniform measure on $(\{0\} \times [0,2\pi])/{\sim}$, and let $H(\cC)$ be the Hilbert space closure of smooth compactly-supported functions on $\cC$ under the Dirichlet inner product.
 Then, as for $\cS$, we define $h_\cC = \sum_i \alpha_i \xi_i$ where $(\xi_i)$ is an orthonormal basis of $H(\cC)$ and $(\alpha_i)$ are independent standard Gaussians. We call $h_\cC$ the GFF on $\cC$ normalized so that $\int_\cC h_\cC\, dm = 0$. 
 
 As before, we can decompose $H(\cC) = H_1(\cC) \oplus H_2(\cC)$ where $H_1(\cC)$ (resp.\ $H_2(\cC)$) is the subspace of functions which are constant (resp. have average zero) on $(\{ t \} \times [0,2\pi])/{\sim}$ for all $t \in \R$. 

\begin{definition}\label{def-2pt-sph}
	Let $W > 0$ and $\alpha = Q - \frac{W}{2\gamma}$. Let 
\[Y_t =
\left\{
\begin{array}{ll}
	B_{t} - (Q -\alpha)t  & \mbox{if } t \geq 0 \\
	\wt B_{-t} +(Q-\alpha) t & \mbox{if } t < 0
\end{array}
\right. , \]
where $(B_s)_{s \geq 0}$ is a standard Brownian motion  conditioned on $B_{s} - (Q-\alpha)s<0$ for all $s>0$, and $(\wt B_s)_{s \geq 0}$ is an independent copy of $(B_s)_{s \geq 0}$. Let $h^1(z) = Y_{\Re z}$ and let $h^2_\cC$ be the projection of an independent GFF $h_\cC$ to $H_2(\cC)$. Set $\wh h = h^1 + h^2_\cC$. Sample an independent real number $\mathbf c$ from the measure $[\frac\gamma2 e^{2(\alpha - Q)c} \, dc]$ on $\R$, and let $\phi = \wh h + \mathbf c$. Let $\cM_2^\sph(W)$ be the infinite measure describing the law of $(\cC, \phi, -\infty, +\infty)/{\sim_\gamma}$. We call a sample from 	$\cM_2^\sph(W)$ a \emph{quantum sphere with two insertions of weight $W$}. 
\end{definition}
The weight $W = 4-\gamma^2$ is special because the two marked points are independent samples from the quantum area measure \cite[Proposition A.13]{wedges}, so the following definition is natural. 

\begin{definition}\label{def-QS}
	Sample $(\cC, \phi, -\infty, +\infty)/{\sim_\gamma}$ from the weighted measure $\mu_\phi(\cC)^{-2} \cM_2^\sph(4-\gamma^2)$. Then we call $(\cC, \phi)/{\sim_\gamma}$ a \emph{quantum sphere}, and denote its law by $\QS$. 
\end{definition}

\begin{remark}
	For $W < \gamma Q$ and compact $I \subset \R$ the $\cM_2^\disk(W)$-mass of the event $\{\text{left boundary length} \in I\}$ is finite (so one can condition on boundary length), but for $W = \gamma Q$ this mass is infinite \cite[Lemma 2.16]{ahs-disk-welding}. The same calculation shows that $\cM_2^\sph(W)[ \text{area }\in I] < \infty$ if and only if $W < 4$. 
\end{remark}

Finally, we will need an area-weighted variant of $\cM_2^\sph(W)$.
\begin{definition}\label{def-sph-weighted}
	Fix $W>0$ and let $(\cC, \phi, -\infty, +\infty)$ be an embedding of a sample from the quantum-area-weighted measure $\mu_\phi(\cC)\cM_2^\sph(W)$. Given $\phi$, sample $\mathbf z$ from the probability measure proportional to $\mu_\phi$. We write $\cM_{2,\bullet}^\sph(W)$ for the law of the marked quantum surface $(\cC, \phi, -\infty, +\infty, \mathbf z)/{\sim_\gamma}$. 
\end{definition}

\subsection{The Liouville field}\label{subsec:LCFT}
In this section we recall the Liouville field which  was constructed in \cite{dkrv-lqg-sphere}.
Let $\exp: \cC \to {\wh \C}$ be the exponential map $z \mapsto e^z$. Let $h_\cC$ be the GFF on the cylinder as defined in the previous section, and let $h_{\wh \C} = h_\cC \circ \exp$. Then $h_{\wh \C}$ is the \emph{GFF on ${\wh \C}$ with average zero on the unit circle}. We write $P_{\wh \C}$ for the law of $h_{\wh \C}$. Its covariance kernel is 
\[G_{\wh \C}(z,w) = -\log|z-w| + \log|z|_++\log|w|+, \qquad |z|_+ := \max (|z|, 1). \]

\begin{definition}\label{def:LF-C}
	Sample $(h, \mathbf c)$ from $P_{\wh \C}\times [e^{-2Qc}\,dc]$ and let $\phi = h - 2Q \log |\cdot|_+ + \mathbf c$. We call $\phi$ the \emph{Liouville field} on ${\wh \C}$ and denote its law by $\LF_{\wh \C}$. 
\end{definition}
For a finite collection of weights $\alpha_i$ and points $z_i \in \C$, we want to define ``$\LF_{\wh \C}^{(\alpha_i, z_i)_i} = \prod_i e^{\alpha_i \phi(z_i)}\LF_{\wh \C}(d\phi)$''. This can be understood via regularization and renormalization, see e.g.\ \cite[Lemma 2.6]{AHS-SLE-integrability}. We give a direct definition below. 
\begin{definition}\label{def:LF-C-insert}
	Let $(\alpha_i, z_i) \in \R \times \C$ for $i = 1, \dots, m$, where $m \geq 1$ and the $z_i$ are distinct. Sample $(h, \mathbf c)$ from $C_{\wh \C}^{(\alpha_i,z_i)_i} P_{\wh \C}\times [e^{(\sum_i \alpha_i - 2Q)c}\,dc]$ where
	\[C_{\wh \C}^{(\alpha_i,z_i)_i} = \prod_{i=1}^m |z_i|_+^{-\alpha_i (2Q-\alpha_i)} e^{\sum_{j=i+1}^m \alpha_i\alpha_j G_{\wh \C}(z_i, z_j)}.\]
	Let $\phi = h - 2Q \log|\cdot|_+ +\sum_{i=1}^m \alpha_i G_{\wh \C}(\cdot, z_i) + \mathbf c$. We call $\phi$ the \emph{Liouville field on ${\wh \C}$ with insertions $(\alpha_i, z_i)$}, and denote its law by $\LF_{\wh \C}^{(\alpha_i, z_i)_i}$. 
\end{definition}

For a conformal automorphism $f: \wh \C \to \wh \C$ and a measure $M$ on the space of distributions on $\C$, let $f_* M$ be the pushforward of $M$ under the map $\phi \mapsto \phi \circ f^{-1}  + Q \log |(f^{-1})'|$. The following change-of-coordinates result is \cite[Theorem 3.5]{dkrv-lqg-sphere} with different notation. We present the version stated in \cite[Proposition 2.29]{AHS-SLE-integrability}.
\begin{proposition}[{\cite[Theorem 3.5]{dkrv-lqg-sphere}}]\label{prop-RV-invariance}
	For $\alpha \in \R$ let $\Delta_\alpha := \frac\alpha2(Q-\frac\alpha2)$. Let $f$ be a conformal automorphism of $\wh \C$ and let $(\alpha_i, z_i) \in \R \times \C$ satisfy $f(z_i)\neq \infty$ for $i = 1, \dots, m$. Then 
	\[ \LF_{\wh \C} = f_* \LF_{\wh \C}, \quad \text{ and } \quad \LF_{\wh \C}^{(\alpha_i, f(z_i))_i} = \prod_{i=1}^m |f'(z_i)|^{-2\Delta_{\alpha_i}} f_* \LF_{\wh \C}^{(\alpha_i, z_i)_i}.\]
\end{proposition}

As the next lemma illustrates, 
sampling points from quantum measures of the Liouville field corresponds to adding insertions to the Liouville field. We recall the proof for the reader's convenience since a closely related argument will be used later. 
\begin{lemma}[{\cite[Lemma 2.31]{AHS-SLE-integrability}}]\label{lem-inserting}
	We have $\mu_\phi(d^2u) \LF_{\wh \C}^{(\alpha_i, z_i)_i}(d\phi) = \LF_{\wh \C}^{(\gamma, u), (\alpha_i, z_i)_i}(d\phi) d^2u$. 
\end{lemma}
\begin{proof}
	Sample $h\sim P_{\wh \C}$.
	Let $h_\eps(u)$ be the average of $h$ on $\partial B_\eps(u)$ and write $G_{{\wh \C}, \eps}(z, u) := \E[h(z)h_\eps(u)]$. Let $f$ be a non-negative continuous function on the Sobolev space $H^{-1}(\C)$, and $g$ a non-negative measurable function on $\R$.  Girsanov's theorem gives
	\[ \E\left[ f(h) \eps^{\gamma^2/2} e^{\gamma h_\eps(u)} \right] = \E\left[f(h + \gamma G_{{\wh \C}, \eps}(\cdot, u)) \right]\E[\eps^{\gamma^2/2} e^{\gamma h_\eps(u)}]. \]
	With $\mu_h^\eps(d^2u) := \eps^{\gamma^2/2} e^{\gamma h_\eps(u)} \,d^2u$ and $\rho_\eps(u) := \E[\eps^{\gamma^2/2} e^{\gamma h_\eps(u)}]$, integrating against $g(u)\,d^2u$ gives
	\[ \E\left[ \int_\C f(h) g(u) \mu_h^\eps (d^2u)\right] = \int_\C \E\left[f(h + \gamma G_{{\wh \C}, \eps}(\cdot, u)) \right]g(u) \rho_\eps(u) \, d^2u.\]
	Taking the $\eps \to 0$ limit yields, with $\rho(u)$ defined by $\rho(u) \, d^2u = \E[\mu_h(d^2u)]$, 
	\[\E\left[\int f(h) g(u)\, \mu_h(d^2u) \right] = \int \E[f(h + \gamma G_{\wh \C}(\cdot, u))] g(u) \rho(u)\, d^2u. \]
	See, e.g., \cite[Section 2.4]{berestycki-lqg-notes} or \cite[Lemma 2.31]{AHS-SLE-integrability} for details on taking this limit.
	
	Let $c \in \R$ and $q(z) = \sum_i \alpha_i G_{\wh \C}(z, z_i) - 2Q \log |z|_+$. For $\wt f$ any non-negative continuous function on $H^{-1}(\C)$ and $\wt g$ any non-negative measurable function on $\R$, choose $f = \wt f(\cdot + q + c)$ and $g = e^{\gamma (q+c)} \wt g$. The above equation, together with $\mu_{h+p} = e^{\gamma p} \mu_h$ for  any continuous function $p: \C \to \R$,  gives
	\[\E\left[\int \wt f(h + q + c)  \wt g(u) \mu_{h+q+c}(d^2u) \right] = \int \E[\wt f( h + \gamma G_{\wh \C}(\cdot, u) + q + c)] \wt g(u) e^{\gamma q(u) +\gamma c}  \rho(u)\, d^2u. \]
	On the other hand, we have 
	$C_{\wh \C}^{(\gamma, u), (\alpha_i, z_i)} = C_{\wh \C}^{(\alpha_i, z_i)_i} C_{\wh \C}^{(\gamma, u)} e^{\gamma q(u) + 2\gamma Q \log |u|_+} = C_{\wh \C}^{(\alpha_i, z_i)_i} e^{\gamma q(u)} \rho(u)$, where the first equality holds by definition and the second follows from a direct calculation; see \cite[Lemma 2.12]{AHS-SLE-integrability} for a similar calculation. Thus, multiplying the previous identity by $C_{\wh \C}^{(\alpha_i, z_i)_i}$ gives 
		\[C_{\wh \C}^{(\alpha_i, z_i)_i}\E\left[\int \wt f(h + q + c)  \wt g(u) \mu_{h+q+c}(d^2u) \right] = \int C_{\wh \C}^{(\gamma, u),(\alpha_i, z_i)_i}\E[\wt f(h + \gamma G_{\wh \C}(\cdot, u) + q + c)] \wt g(u) e^{\gamma c} \, d^2u. \]
		Multiplying by $e^{(\sum_i \alpha_i - 2Q)c}$ and integrating over $c$ gives
		\[\LF_{\wh \C}^{(\alpha_i, z_i)_i} \left[ \int_\C \wt f(\phi) \wt g(u) \mu_\phi(d^2u) \right] = \int_\C \LF_{\wh \C}^{(\gamma, u), (\alpha_i, z_i)_i}[\wt f(\phi)] \wt g(u) \, d^2u. \]
		The functions $\wt f, \wt g$ are arbitrary so the desired result holds. 
\end{proof}

We need the following Liouville field description of  $\cM_{2,\bullet}^\sph(W)$.
\begin{proposition}\label{prop-3pt-QS-LF}
	Fix $W>0$, let $\alpha = Q - \frac{W}{2\gamma}$ and sample $\phi$ from $\frac{2\pi \gamma}{(Q-\alpha)^2}\LF_{\wh \C}^{(\alpha, 0), (\alpha, 1), (\gamma, -1)}$.  Then the law of $({\wh \C}, \phi, 0, 1, -1)/{\sim_\gamma}$ is $\cM_{2,\bullet}^\sph(W)$. 
\end{proposition}
\begin{proof}
	\cite[Proposition B.7]{AHS-SLE-integrability} describes the field of $\cM_{2,\bullet}^\sph(W)$ in terms of the Liouville field on $\cC$, then~\cite[Lemma B.4]{AHS-SLE-integrability} gives the coordinate change from $\cC$ to ${\wh \C}$. 
\end{proof}

\subsection{Uniform embedding of quantum surfaces}\label{subsec-haar}

Let $\conf(\wh\C)$ denote the space of automorphisms of the Riemann sphere $\wh \C = \C \cup \{\infty\}$. Being a locally compact Lie group, it has a right-invariant Haar measure which is unique modulo multiplicative constant, and since $\conf(\wh \C)$ is unimodular the measure is also left-invariant. Let $\mathbf m_{\wh \C}$ be such a Haar measure. 
The following gives an explicit description of $\mathbf m_{\wh \C}$; see e.g.\ \cite[Lemma 2.28]{AHS-SLE-integrability}.
\begin{lemma}\label{lem-haar}
	Let $\mathfrak f$ be sampled from $\mathbf m_{\wh \C}$. Then there is a constant $C \in (0,\infty)$ such that the law of $(\mathfrak f(0), \mathfrak f(1), \mathfrak f(-1))$ is $C |(p-q)(q-r)(r-p)|^{-2}\, d^2p \, d^2q \, d^2r$. 
\end{lemma}

Suppose $M$ is a measure on the space of quantum surfaces which can be embedded in $\wh \C$. Sample a pair $(\mathfrak f,S)$ from the product measure $\mathbf m_{\wh \C} \times M$, and let $\phi_0$ be a distribution on $\C$ chosen in a way measurable with respect to $S$ such that $S = (\wh \C, \phi_0)/{\sim_\gamma}$. We define $\mathbf m_{\wh \C} \ltimes M$ to be the law of $\mathfrak f \bullet_\gamma \phi_0$. We call $\mathbf m_{\wh \C} \ltimes M$ the \emph{uniform embedding of $M$}. Note that the definition of uniform embedding does not depend on the choice of $\phi_0$. Recall that $\QS$ is the law of the quantum sphere from Definition~\ref{def-QS}.

\begin{proposition}[{\cite[Theorem 1.2]{AHS-SLE-integrability}}]\label{prop-unif-embed}
	There is a constant $C$ such that $\mathbf m_{\wh \C} \ltimes \QS = C \cdot \LF_{\wh \C}$. 
\end{proposition}

\subsection{Conformal welding}\label{subsec:welding}

Let $\kappa>0$ and let $(D,p,q)$ be a simply-connected domain with two marked boundary points.
$\SLE_\kappa$ is a conformally invariant random curve in $D$ from $p$ to $q$ introduced by Schramm \cite{schramm-sle}, which describes the scaling limits of many statistical physics models. When $\kappa <4$, almost surely $\SLE_\kappa$ is simple and only intersects $\partial D$ at $\{p,q\}$. We will also need a spherical variant of $\SLE$: for distinct points $p,q \in \C$ and $\rho > -2$, there is a random curve from $p$ to $q$ called \emph{whole-plane $\SLE_\kappa(\rho)$}, see e.g.\ \cite[Section 2.1.3]{ig4} for its definition. 
 
 For $\kappa \in (0,8)$ and distinct points $p,q \in \C$, the \emph{two-sided whole-plane SLE}, which we denote by $\SLE^{p \rightleftharpoons q}_\kappa$, is the probability measure on pairs of curves $(\eta_1,\eta_2)$ on $\C$ connecting $p$ and $q$ where $\eta_1$ is a whole-plane $\SLE_\kappa(2)$ from $p$ to $q$, and conditioning on $\eta_1$, the curve $\eta_2$ is a chordal $\SLE_\kappa$ on the complement of $\eta_1$  from $q$ to $p$. This pair of curves $(\eta_1,\eta_2)$ satisfies the following resampling property: conditioning on one, the other has the law of 
 chordal $\SLE_\kappa$ in the complement, see e.g.\ \cite[Section 2.2]{zhan-loop-measures}.

We need a special case of \cite[Theorem 2.4]{ahs-disk-welding}. Let $\wh \C = \C\cup \{ \infty\}$ be the Riemann sphere. Let $\{\cM_2^\disk (2; \ell_1, \ell_2)\}_{\ell_1, \ell_2}$ be a disintegration of $\cM_2^\disk(2)$ on its two boundary arc lengths, i.e.\ $\cM_2^\disk(2) = \iint \cM_2^\disk(2;\ell_1, \ell_2)\,d\ell_1\,d\ell_2$, and a sample from $\cM_2^\disk(2;\ell_1, \ell_2)$ a.s.\ has boundary lengths $(\ell_1, \ell_2)$. 
\begin{proposition}\label{prop:AHS-orig}
	Fix distinct $p,q \in \C$ and let $(\wh \C, \phi, p, q)$ be an embedding of a sample from $\cM_2^\sph(4)$. Independently sample $(\eta_1, \eta_2)$ from the probability measure  $\SLE^{p \rightleftharpoons q}_\kappa$, and let $D_1$ and $D_2$ be the connected components of $\wh \C \backslash (\eta_1 \cup \eta_2)$ lying to the left and right of $\eta_1$ respectively. Then there is a constant $C$ such that the joint law of $(D_1, \phi, p,q)/{\sim_\gamma}$ and $(D_2, \phi, p,q)/{\sim_\gamma}$ is
	\[ C \int_0^\infty\int_0^\infty \cM_2^\disk(2; \ell_1, \ell_2) \times \cM_2^\disk(2; \ell_2, \ell_1) \, d\ell_1 \, d\ell_2. \]
\end{proposition}
The above statement of Proposition~\ref{prop:AHS-orig} is in terms of cutting a sphere to get two disks. It can be equivalently expressed in terms of gluing two disks to get a loop-decorated sphere. 
For fixed $\ell_1, \ell_2>0$, a pair of quantum disks sampled from $ \cM_2^\disk(2; \ell_1, \ell_2) \times \cM_2^\disk(2; \ell_2, \ell_1) $ can be \emph{conformally welded} along their boundary arcs according to quantum boundary length, to get a quantum surface with the sphere topology decorated by two points and a loop passing through them. 
For more details on the conformal welding of quantum surfaces, see e.g.\  \cite{shef-zipper, wedges, ghs-mating-survey}, and see \cite{ahs-disk-welding} for more information on the conformal welding of quantum disks. 

We now give a more precise definition of the measure $\mathrm{Weld}(\QD,\QD)$ appearing in Theorem \ref{thm-loop}. Let $\ell>0$ and let $(\cD_1, \cD_2) \sim \QD(\ell) \times \QD(\ell)$. For $i=1,2$, let $\phi_i: \mathbb S^1_\ell \rightarrow \cD_i$ be a parametrization of the boundary of $\cD_i$ according to its quantum  boundary length such that $\phi_i$ traces the boundary in counterclockwise direction when the disk is embedded in $\BB D$. Namely, for $0<s<t<1$, $\phi_i([s,t])$ is an arc on the boundary of $\cD_i$ with quantum length $t-s$, 
where we represent  $\mathbb S^1_\ell$ as the interval $[0,\ell]$ with endpoints identified.  Let $U$ be a uniform point on $\mathbb S^1_\ell$ independent of everything else. Let $\mathrm{Weld}(\cD_1,\cD_2)$ be the  curve-decorated quantum surface obtained by conformally welding $\cD_1$ and $\cD_2$ along their boundaries where $\phi_1(t)$ is identified with $\phi_2(U-t)$ for all $t\in \mathbb S^1_\ell$. In words, $\mathrm{Weld}(\cD_1,\cD_2)$ means we conformally weld $\cD_1$ and $\cD_2$ according to their boundary length uniformly at random. Let $\mathrm{Weld}(\QD(\ell), \QD(\ell))$ be the law of $\mathrm{Weld}(\cD_1, \cD_2)$, and define $\mathrm{Weld}(\QD,\QD) = \int_0^\infty \ell \cdot \mathrm{Weld}(\QD(\ell), \QD(\ell))$.

\subsection{Zhan's construction of the SLE loop measure}\label{sec-zhan-loop-measure}
Given a simple loop $\eta$ and some $d\in[0,2]$, let $\cont_{\eta,\eps}$ be $\eps^{d-2}$ times Lebesgue area measure restricted to the $\eps$-neighborhood of $\eta$.
	If $\lim\limits_{\eps\to0} \cont_{\eta,\eps}$ exists for the weak topology then we denote the limit by $\cont_\eta$ and call it the \emph{$d$-dimensional Minkowski content} of $\eta$.

We can view $\SLE^{p \rightleftharpoons q}_\kappa$ as a measure on oriented loops by concatenating $\eta_1$ and $\eta_2$. Given a loop $\eta$ sampled from $\SLE^{p \rightleftharpoons q}_\kappa$, with probability 1 the dimension of $\eta$ is $d=1+\frac{\kappa}{8}$~\cite{beffara-dim} and its $d$-dimensional Minkowski content $\cont_\eta$ exists~\cite{lawler-rezai-nat}.  The \emph{(unrooted) SLE loop measure} $\SLE_\kappa^\mathrm{loop}$ on $\C$ is an infinite measure on oriented loops defined by (see \cite[Theorem 4.2]{zhan-loop-measures}) 
\begin{equation}\label{eq:zhan-loop}
\SLE_\kappa^\mathrm{loop}(d\eta)= |\cont_\eta|^{-2} \iint_{\C\times \C}  |p-q|^{-2(2-d)} \SLE^{p \rightleftharpoons q}_\kappa(d\eta) \,d^2p\,d^2q.
\end{equation}
The operation of forgetting $p$ and $q$ in~\eqref{eq:zhan-loop} is natural, since given $\eta$, the points $p,q$ are conditionally independent points sampled from the Minkowski content measure $\mathrm{Cont}_\eta$ on $\eta$; precisely, \cite[Theorem 4.2 (i)]{zhan-loop-measures} states
\begin{equation}\label{eq:zhan-loop-pts}
\SLE_\kappa^\mathrm{loop} (d\eta)\, \mathrm{Cont}_\eta(dp)\,\mathrm{Cont}_\eta(dq) = |p-q|^{\frac{\gamma^2}4-2} \,\SLE_\kappa^{p \rightleftharpoons q}(d\eta) \, d^2p \, d^2q.
\end{equation}
For $\kappa\in (0,4]$, Zhan \cite{zhan-loop-measures} shows that $\SLE_\kappa^\mathrm{loop}$
is an example of a Malliavin-Kontsevich-Suhov (MKS) loop measure.

\subsection{The Gromov-Hausdorff-Prokhorov-uniform metric} 
\label{sec:ghpu}
In this subsection we will define precisely the space of compact loop-decorated metric measure spaces and the Gromov-Hausdorff-Prokhorov-uniform metric. This, along with definitions in Section \ref{sec-loop}, will make precise the statement of Theorem \ref{thm:saw-loop}. We remark that the analogous definitions in the setting of curve-decorated metric measure spaces were first made in \cite{gwynne-miller-uihpq}; see also \cite{gromov-metric-book,bbi-metric-geometry,gpw-metric-measure,miermont-survey,adh-ghp}.

For a metric space $(X,d)$ let $C_0(X)$ denote the space of parametrized and oriented loops on $X$ with no distinguished starting point. More precisely, identifying the circle $\mathbb S^1_\ell$ of length $\ell$ with the interval $[0,\ell]$ with endpoints identified, $C_0(X)$ is the space of continuous functions $\eta:\mathbb S^1_\ell\to X$, where we identify $\eta$ and $\wt\eta$ if $\wt \eta(t) = \eta(t+r)$ for some $r \in [0,\ell)$ and all $t \in [0,\ell)$. Let $\BB d_d^{\op{H}}$ denote the $d$-Hausdorff metric on compact subsets of $X$ and let $\BB d_d^{\op{P}}$ denote the $d$-Prokhorov metric on finite measures on $X$. Finally, let $\BB d_d^{\op{U}}$ denote the $d$-uniform metric on $C_0(X)$, i.e., 
$$
\BB d_d^{\op{U}}(\eta,\wt\eta) = \inf_{r\in[0,\ell)} \sup_{t\in[0,\ell)} d( \eta(t),\wt\eta(t+r) ).
$$

Let $\BB M^{\op{GHPU}}$ be the set of compact \emph{loop-decorated metric measure spaces}, i.e., $\BB M^{\op{GHPU}}$ is the set of 4-tuples $\frk X=(X,d,\mu,\eta)$ where $(X,d)$ is a compact metric space, $\mu$ is a finite Borel measure on $X$, and $\eta \in C_0(X)$. Given elements 
$\frk X_1=(X_1,d_1,\mu_1,\eta_1)$ and 
$\frk X_2=(X_2,d_2,\mu_2,\eta_2)$ of $\BB M^{\op{GHPU}}$, we define their Gromov-Hausdorff-Prokhorov-uniform (GHPU) distance by
$$
\BB d^{\op{GHPU}}(\frk X_1,\frk X_2)
= \inf_{(W,d),\iota_1,\iota_2} 
\BB d_D^{\op{H}}( \iota_1(X_1),\iota_2(X_2) ) + 
\BB d_D^{\op{P}}( (\iota_1)_*\mu_1, (\iota_2)_*\mu_2 ) +
\BB d_D^{\op{U}}(\iota_1\circ\eta_1,\iota_2\circ\eta_2),
$$ 
where we take the infimum over all compact metric spaces $(W,D)$ and isometric embeddings $\iota_1:X_1\to W$ and $\iota_2:X_2\to W$. It is shown in \cite{gwynne-miller-uihpq} that this defines a complete separable metric in the setting of curve-decorated (rather than loop-decorated) metric measure spaces if we identify two elements of this space which differ by a measure- and curve-preserving isometry. The analogous statement holds in the setting of loop-decorated metric measure spaces since we obtain a loop by considering a curve that forms a loop and identifying two such curves which differ by a time shift.

\section{The SLE loop via conformal welding: proof of Theorem~\ref{thm-loop}}\label{sec-loop}

In Section~\ref{subsec:proof} we prove Theorem~\ref{thm-loop} modulo  Proposition~\ref{prop:M3}, whose proof is given in Section~\ref{subsec:loop-haar}.

\subsection{Proof of Theorem~\ref{thm-loop}}\label{subsec:proof}

Fix $\gamma\in (0,2)$ and $\kappa=\gamma^2\in (0,4)$. 
Let $M=\QS\otimes \SLE_{\kappa}^{\op{loop}}$ be the law of the loop-decorated quantum surface called the MKS-loop-decorated quantum sphere with parameter $\gamma$.
Namely, if $(S, \eta)$ is sampled from $\QS \times  \SLE_{\kappa}^{\op{loop}}$ and $\phi$ is a distribution on $\C$ chosen in a way measurable with respect to $S$ such that $S = (\wh \C, \phi)/{\sim_\gamma}$, then $M$ is the law of the loop-decorated quantum surface $(\wh \C, \phi, \eta)/{\sim_\gamma}$.
Theorem~\ref{thm-loop} asserts that   $M=C\mathrm{Weld}(\QD,\QD)$ for some constant $C\in (0,\infty)$. 
To prove it,   we need a variant of Proposition~\ref{prop:AHS-orig} where the marked points on the quantum disks are forgotten.
\begin{lemma}\label{lem:AHS}
	Let $(\wh \C,h,p,q)$ be an embedding of a sample from $\cM^\sph_2(4)$. Let $(\eta_1, \eta_2)$ be a sample from $\SLE^{p \rightleftharpoons q}_\kappa$ independent of $h$, and let $\eta$ be the oriented loop obtained by concatenating $\eta_1$ and $\eta_2$. 
	Then viewed as a loop-decorated quantum surface the law of $(\wh \C,h,\eta)$ equals 
	\(C\int_0^\infty \ell^3\cdot  \mathrm{Weld}(\QD(\ell),\QD(\ell)) \, d\ell\) for some constant $C\in (0,\infty)$.
\end{lemma}
\begin{proof}
		Let $F$ be the map that forgets the marked points of a quantum surface. By Definition~\ref{def-QD}, 
		\alb
		\int_0^\infty \ell^2 \QD(\ell)\, d\ell = F_*\cM_2^\disk(2) &= F_* \int_0^\infty\int_0^\infty \cM_2^\disk(2; \ell_1, \ell_2)\, d\ell_1\,d\ell_2 \\ &= \int_0^\infty \int_0^\ell F_* \cM_2^\disk (2;\ell_1, \ell - \ell_1)\, d\ell_1 \, d\ell.
		\ale
		In the last equality, we change variables $\ell = \ell_1 +\ell_2$ so $1_{\ell_1, \ell_2 >0} \, d\ell_1 \, d \ell_2$ corresponds to $1_{\ell > \ell_1>0} \, d\ell_1 \, d \ell$.
		By Proposition~\ref{prop-QD-typical}, the measure $F_* \cMtwo(2; \ell_1, \ell - \ell_1)$ does not depend on the choice of $\ell_1$, and hence must equal $\ell \QD(\ell)$. 
		
		Let $D_1$ and $D_2$ be the connected components of $\wh \C \backslash \eta$. 
		By Proposition~\ref {prop:AHS-orig}, the law of the pair of marked quantum surfaces $((D_1,h, p,q)/{\sim_\gamma}, (D_2,h, p,q)/{\sim_\gamma})$ equals 
		\[C \int_0^\infty \int_0^\ell \cMtwo(2; \ell_1,\ell-\ell_1) \times \cMtwo(2; \ell - \ell_1, \ell_1) \, d\ell_1 \, d\ell.\]
		Applying $F$ to both sides and using $F_* \cMtwo(2; \ell_1, \ell - \ell_1) = \ell \QD(\ell)$, the law of $((D_1, h)/{\sim_\gamma}, (D_2, h)/{\sim_\gamma})$ is $C \int_0^\infty \ell^3 \QD(\ell)^2\, d\ell$. 
		Finally, since the conformal welding of $(D_1, h,p,q)/{\sim_\gamma}$ and $(D_2, h,p,q)/{\sim_\gamma}$ is determined by the locations of the marked points, and the marked points on each disk are uniformly chosen from quantum length measure (Proposition~\ref{prop-QD-typical}),
the conformal welding of $(D_1, h)/{\sim_\gamma}$ to $(D_2, h)/{\sim_\gamma}$ is uniform, as desired.
\end{proof}

Similarly as for the proof of Proposition~\ref{prop-unif-embed} from \cite{AHS-SLE-integrability}, we prove Theorem~\ref{thm-loop} by adding three marked points. 
Suppose $(\wh\C, h,\eta)$ is an embedding of a  sample from $M$ weighted by $\mu_h(\C)$ times the square of the quantum length of $\eta$. Given $(h,\eta)$,
independently  sample $p,q$ from the probability measure proportional to the quantum length measure on $\eta$, and $r$ from the probability measure proportional to the quantum area measure, so $p, q \in \eta$ and $r \in \C$. Let $M_3$ be the law of $(\wh\C, h, \eta, p,q,r)$ viewed as a loop-decorated quantum surface with three marked points. 
Recall that $\SLE_\kappa^{p \rightleftharpoons q}$ is the law of a two-sided whole plane $\SLE_\kappa$ from $p$ to $q$.
Moreover, we view a sample $(\eta_1,\eta_2)$ from $\SLE_\kappa^{p \rightleftharpoons q}$ as an oriented loop by concatenating $\eta_1$ with $\eta_2$.
Recall   $\cM^\sph_{2,\bullet}(W)$ from Definition~\ref{def-sph-weighted}.
The following proposition describes 
$M_3$ in terms of  $\cM^\sph_{2,\bullet}(W)$ and $\SLE_\kappa^{p \rightleftharpoons q}$.
\begin{proposition}\label{prop:M3}
	Let $(\wh \C, h, p,q,r)$ be an embedding of  a sample from $\cM^\sph_{2,\bullet}(4)$.
	Independently sample $\eta$ from  $\SLE_\kappa^{p \rightleftharpoons q}$. Let $\wt M_3$ be the law of $(\wh\C, h, \eta, p,q,r)$ viewed as a loop-decorated quantum surface with three marked points. 
	Then there exists a constant $C\in (0,\infty)$ such that $M_3=C\wt M_3$.
\end{proposition}

\begin{proof}[Proof of Theorem~\ref{thm-loop} given Proposition~\ref{prop:M3}]
	See Figure~\ref{fig-M3}. 
	Fix $p, q, r \in \C$. Sample a decorated quantum surface from $\wt M_3$ and embed it as $(\wh \C, \phi, \eta, p,q,r)$. Let $(A, L)$ be its quantum area and the quantum length of its loop. By Definition~\ref{def-sph-weighted}, after weighting by $A^{-1}$ the law of $(\wh \C, \phi, \eta, p,q)$ is $\cM^\sph_{2}(4) \otimes \SLE_\kappa^{p \rightleftharpoons q}$, then by Lemma~\ref{lem:AHS} the law of $(\wh \C, \phi, \eta)$ is $C\int_0^\infty \ell^3 \mathrm{Weld}(\QD(\ell), \QD(\ell))\, d\ell$ for some $C>0$. Further weighting by $L^{-2}$, the law of $(\wh \C, \phi, \eta)$ is $C\int_0^\infty \ell \mathrm{Weld}(\QD(\ell), \QD(\ell))\, d\ell = C\mathrm{Weld}(\QD, \QD)$.
	
	By the definition of $M_3$, if we embed a sample from $M_3$ as $(\wh \C, \phi, \eta, p,q,r)$ and let $(A, L)$ be its quantum area and the quantum length of its loop, then the law of $(\wh \C, \phi, \eta)$ after weighting by $A^{-1}L^{-2}$ is $\QS \otimes \SLE_\kappa^\mathrm{loop}$. 
	
	Proposition~\ref{prop:M3} states that $M_3$ and $\wt M_3$ agree up to multiplicative constant, so by the above two paragraphs $\QS \otimes \SLE_\kappa^\mathrm{loop}$ and $\mathrm{Weld}(\QD,\QD)$ agree up to multiplicative constant.
\end{proof} 

\begin{figure}[ht!]
	\begin{center}
		\includegraphics[scale=0.37]{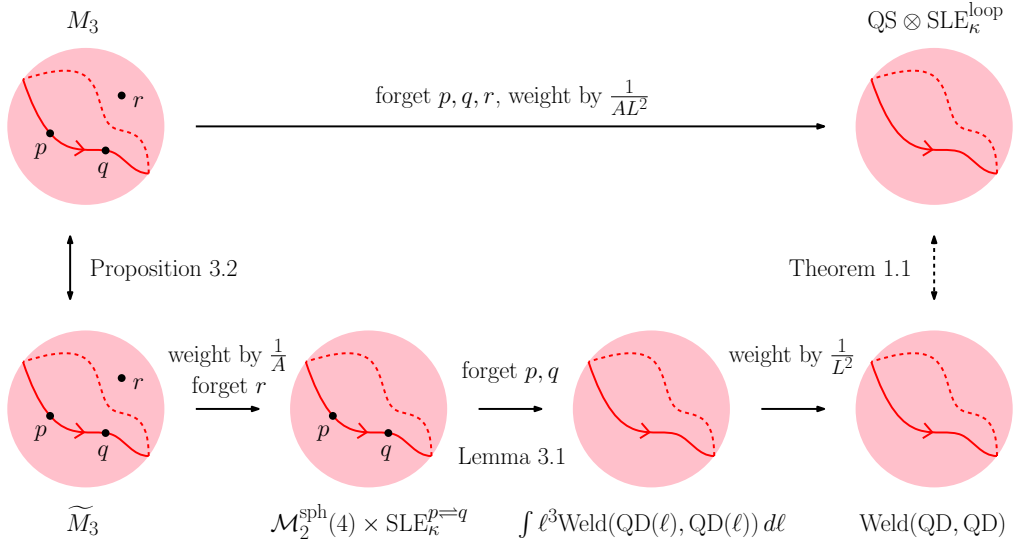}
	\end{center}
	\caption{\label{fig-M3}   Illustration for proof of Theorem~\ref{thm-loop}. Measures are displayed without multiplicative constants. We denote the quantum area by $A$ and the quantum length of the loop by $L$.}
\end{figure}

\subsection{Proof of Proposition~\ref{prop:M3} via the uniform embedding}\label{subsec:loop-haar}
We will prove Proposition~\ref{prop:M3}  by first establishing Proposition \ref{prop:M3-haar}, which gives its counterpart under the uniform embedding.
As for $\QS$ in Section~\ref{subsec-haar},
 	suppose we sample $(\mathfrak f, (\wh\C, h,\eta,0,1,-1)/{\sim_\gamma})$ from $\haar_{\wh\C} \times M_3$.
The uniform embedding of $M_3$ via $\haar_{\wh \C}$, which we denote by $\haar_{\wh \C} \ltimes M_3$, 
is the law of $ (\mathfrak f\bullet_\gamma h, \mathfrak f\circ \eta, \mathfrak f(0),\mathfrak f(1), \mathfrak f(-1))$. We can similarly define $\haar_{\wh \C} \ltimes M$ and $\haar_{\wh \C} \ltimes \wt M_3$.
\begin{proposition}\label{prop:M3-haar}
	There exists a  constant $C\in (0,\infty)$ such that  $\haar_{\wh \C} \ltimes M_3= C\haar_{\wh \C}\ltimes \wt M_3$.
\end{proposition}
We first  give the uniform embedding of $M$.
\begin{lemma}\label{lem:M-haar}
	There exists a  constant $C\in (0,\infty)$ such that $\haar_{\wh \C} \ltimes M=C\cdot \LF_{\wh \C} \times \SLE^{\mathrm{loop}}_\kappa$.
\end{lemma}
\begin{proof}
	The measure $\SLE^{\mathrm{loop}}_\kappa$ is conformally invariant, namely, for each $f\in \conf(\wh \C)$, 
	the law of $f\circ \eta$ is $\SLE^{\mathrm{loop}}_\kappa$ if $\eta$ is sampled from $\SLE^{\mathrm{loop}}_\kappa$.
	Now Lemma~\ref{lem:M-haar} follows from Proposition~\ref{prop-unif-embed}.
\end{proof}
We now describe the uniform embedding of $M_3$. 
\begin{lemma}\label{lem:M3-haar}
	There exists a constant  $C\in (0,\infty)$ such that 
	\[
	\haar_{\wh \C} \ltimes  M_3 = C|p-q|^{\frac{\gamma^2}4-2} \,\LF_{\wh\C}^{(\frac\gamma2, p),(\frac\gamma2, q), (\gamma, r)} (d\phi)\, \SLE_\kappa^{p \rightleftharpoons q} (d\eta) \,d^2 p \,d^2q\, d^2r.
	\]
\end{lemma}

To prove Lemma~\ref{lem:M3-haar} we use an analog of  Lemma~\ref{lem-inserting} based on the Girsanov theorem. We first review some background on the Minkowski content of SLE and its relation to  quantum length.
As before we denote the  $(1+\frac\kappa8)$-dimensional Minkowski content  of an $\SLE_\kappa$-type curve $\eta$ by $\mathrm{Cont}_\eta$.
\begin{lemma}\label{lem:cont0}
	Let $d=1+\frac{\kappa}{8}$. Let $\eta$ be sampled from $\SLE_\kappa^\lp$. Then almost surely 
	\begin{equation}\label{eq:energy}
		\int_{\C^2}  \frac{\cont_\eta(dx)\,  \cont_\eta(dy) }{|x-y|^{d-\eps}}<\infty \quad \textrm{for each }\eps\in (0,d).
	\end{equation}	
\end{lemma}
\begin{proof}
	By 
	Green's function estimates for chordal SLE (see e.g.\ \cite{lawler-rezai-nat})
	\eqref{eq:energy} holds if $\eta$ is sampled from a chordal $\SLE_\kappa$ even after we take expectation over the integral.
	By local absolutely continuity, \eqref{eq:energy} holds for $\SLE_\kappa^\lp$.
\end{proof} 
For each $\eta$ such that~\eqref{eq:energy} holds, the Gaussian multiplicative chaos (GMC)  measure (see e.g.\ \cite{berestycki-gmt-elementary})
\[
\nu^\eta_h:= \lim_{\eps\to0}\eps^{\frac{\gamma^2}{8}}e^{\frac12\gamma h_\eps}\cont_\eta
\]
exists,
where $h$ is sampled from the Gassian free field measure $P_\C$. 
By \cite[Section 3.2]{benoist-lqg-chaos},  modulo a multiplicative constant, $\nu^\eta_h$ is the quantum length of $\eta$ with respect to $h$.
We now give an analog of Lemma~\ref{lem-inserting}.
\begin{lemma}\label{lem:cont}
	Suppose $\eta$ is a loop satisfying~\eqref{eq:energy}. 
	For $\alpha\in \R$ and $z\in\C$, we have 
	\[
	\nu^\eta_{\phi}(du)\,\LF^{(\alpha,z)}_{\wh \C}(d\phi)  =
	\LF_{\wh\C}^{(\alpha,z),(\frac\gamma2, u)}(d\phi)\,\mathrm{Cont}_\eta(du).
	\]
\end{lemma}
\begin{proof}
	The proof is identical to that of Lemma~\ref{lem-inserting}, except we replace the quantum area measure $\mu_\phi(d^2 u) = \lim_{\eps \to 0}\eps^{\gamma^2/2} e^{\gamma \phi_\eps(u)} \, d^2u$ with the GMC measure $\nu^\eta_\phi(du) = \lim_{\eps\to0}\eps^{\frac{\gamma^2}{8}}e^{\frac12\gamma \phi_\eps}\cont_\eta(du)$. 
\end{proof}

\begin{proof}[Proof of Lemma~\ref{lem:M3-haar}]
	Since $\nu^\eta_{\phi}$ is the quantum length measure on $\eta$ modulo a multiplicative constant, 
	by Lemma~\ref{lem:M-haar},   there exists $C\in(0,\infty)$ such that
	\[
	\haar_{\wh \C}\ltimes M_3= C\mu_\phi(dr)\,\nu^\eta_\phi(dp)\,\nu^\eta_\phi(dq)\,\LF_{\wh \C}(d\phi)\, \SLE_\kappa^\lp(d\eta).
	\]	
	By Lemma~\ref{lem:cont0} $\eta$ almost surely satisfies~\eqref{eq:energy}. Thus, applying Lemma~\ref{lem:cont} twice, we get 
	\eqbn
	\begin{split}
		&\nu^\eta_\phi(dp)\,\nu^\eta_\phi(dq)\,\LF_{\wh \C}(d\phi) \, \SLE_\kappa^\lp(d\eta)\\
		&\qquad= 
		\LF_{\wh\C}^{(\frac\gamma2,p),(\frac\gamma2, q)}(d\phi)\, \mathrm{Cont}_\eta(dp)\,\mathrm{Cont}_\eta(dq)\, \SLE_\kappa^\mathrm{loop}(d\eta).
	\end{split}
	\eqen
	By Lemma~\ref{lem-inserting}, we get further that \(\mu_\phi(dr)\,\nu^\eta_\phi(dp)\,\nu^\eta_\phi(dq)\,\LF_{\wh \C}(d\phi)\,  \SLE_\kappa^\lp(d\eta)\) equals 
	\[
	\LF_{\wh\C}^{(\frac\gamma2,p),(\frac\gamma2, q),(\gamma,r) }(d\phi)\, \mathrm{Cont}_\eta(dp)\,\mathrm{Cont}_\eta(dq)\, \SLE_\kappa^\mathrm{loop}(d\eta)\,d^2r.
	\]
	Comparing against~\eqref{eq:zhan-loop-pts} completes the proof. 
\end{proof}

We now switch our attention to $\haar\ltimes \wt M_3$. The following lemma describes the embedding of $\wt M_3$.
\begin{lemma}\label{lem:M3-embed}
	Given distinct $p,q,r$ on $\wh \C$, let $\wt M^{p,q,r}_3$ be the law of $(\phi, \eta)$ where $(\wh \C, \phi, \eta,p,q,r)$ is an embedding of a  sample from $\wt M_3$.
	Then there exists a constant $C\in (0,\infty)$	such that 
\[
\wt M^{p,q,r}_3= C |p-q|^{\frac{\gamma^2}{4}-2} |(p-q)(q-r)(r-p)|^2 \LF_{\wh \C}^{(\frac\gamma2, p), (\frac\gamma2, q), (\gamma, r)} \times \SLE_\kappa^{p\rightleftharpoons q}.
\]	
\end{lemma}
\begin{proof}
	By  Proposition~\ref{prop-3pt-QS-LF}, there exists a constant $C\in (0,\infty)$ such that
	\[
	\wt M^{0,1,-1}_3= C \LF_{\wh\C}^{(\frac\gamma2, 0
		),(\frac\gamma2, 1), (\gamma, -1)} \times \SLE_\kappa^{0\rightleftharpoons 1}.
	\]
	Suppose $f\in \conf(\wh\C)$ maps $(0,1,-1)$ to $(p,q,r)$; explicitly, we have $f(z) = \frac{(pq-2qr+rp)z + p(q-r)}{(2p-q-r)z + q-r}$, and 
	\[ f'(0) = \frac{2(p-q)(q-r)(r-p)}{(q-r)^2}, \quad f'(1) = \frac{2(p-q)(q-r)(r-p)}{4(r-p)^2}, \quad f'(-1) = \frac{2(p-q)(q-r)(r-p)}{4(p-q)^2}.\]
	We have 
	\[
	f'(0)f'(1)=4(p-q)^2 \quad \textrm{and} \quad f'(0)f'(1)f'(-1)=2(p-q)(q-r)(r-p).
	\]
	By Proposition~\ref{prop-RV-invariance} the field of $\wt M_3^{p,q,r}$ is given by
	\[
	f_* \LF_{\wh \C}^{(\frac{\gamma}2, 0), (\frac{\gamma}2, 1), (\gamma,-1)} = 
	|f'(0)|^{2\Delta_{\frac{\gamma}{2}}}  
	|f'(1)|^{2\Delta_{\frac{\gamma}{2}}}  
	|f'(-1)|^{2\Delta_{\gamma}}     \LF_{\wh \C}^{(\frac{\gamma}2, p), (\frac{\gamma}2, q), (\gamma, r)},
	\]
	where $\Delta_\alpha = \frac\alpha2(Q - \frac\alpha2)$.  Since $\Delta_{\frac\gamma2}=\frac12+\frac{\gamma^2}{16}$ and $\Delta_\gamma = 1$, we get the desired result.
\end{proof}

\begin{proof}[Proof of Proposition~\ref{prop:M3-haar}]
	
	By Lemma~\ref{lem-haar} and the definition of $\wt M^{p,q,r}_3$ in Lemma~\ref{lem:M3-embed},  we see that
	\begin{equation}\label{eq:M3-haar}
		\haar_{\wh \C}\ltimes \wt M_3=C\wt M^{p,q,r}_3 |(p-q)(q-r)(r-p)|^{-2} d^2p\, d^2q\, d^2r \textrm{ for some }C\in (0,\infty). 
	\end{equation}
	Now Lemmas~\ref{lem:M3-haar} and~\ref{lem:M3-embed} together give $\haar_{\wh \C}\ltimes M_3=C\haar_{\wh \C}\ltimes \wt M_3$ for a possibly different constant $C$.
\end{proof}

\begin{proof}[Proof of Proposition~\ref{prop:M3}]
	Given distinct $p,q,r$ on $\wh \C$, let $M^{p,q,r}_3$ be the law of $(\phi, \eta)$ where $(\wh \C, \phi, \eta,p,q,r)$ is a sample from $M_3$.
	By the definition of uniform embedding, the law of $(\phi, \eta)$ sampled from $\haar_{\wh\C}\ltimes M_3$ agrees with that of $(\mathfrak f \bullet_\gamma \phi_0, \mathfrak f \circ \eta_0)$ where $(\mathfrak f, \phi_0, \eta_0) \sim \haar_{\wh \C} \times M_3^{0,1,-1}$. The $\haar_{\wh \C}$-law of $\mathfrak f$ is described by Lemma~\ref{lem-haar}, and by definition, if $f$ is the conformal automorphism of $\wh \C$ sending $(0,1,-1)$ to $(p,q,r)$ and $(\phi_0, \eta_0) \sim M_3^{0,1,-1}$, the law of $(f \bullet_\gamma \phi_0, f \circ \eta_0)$ is $M_3^{p,q,r}$.
	Thus~\eqref{eq:M3-haar} holds with $M_3$  and $M^{p,q,r}_3$ in place of $\wt M_3$  and $\wt M^{p,q,r}_3$.
Consequently
	\[
	M^{p,q,r}_3 d^2p\, d^2q \, d^2r=C   \wt M^{p,q,r}_3d^2p \, d^2q \, d^2r \quad \textrm{for some }C\in (0,\infty).
	\]
	This gives   $M^{p,q,r}_3  =C \wt M^{p,q,r}_3 $ for almost every $p,q,r$. Using any such $p,q,r$, we conclude $M_3=C\wt M_3$ as desired.
\end{proof}

\begin{remark}[KPZ relation]
	As seen in the proof of Lemma \ref{lem:M3-haar}, a crucial fact to our proof is that the exponent $\frac{\gamma^2}{4}-2$ is equal to 
	$-2(2-d)=\frac{\kappa}{4}-2$ from~\eqref{eq:zhan-loop} where $d=1+\frac{\kappa}8$  is the dimension of $\SLE_\kappa$. 
	As seen in the proof of Lemma \ref{lem:cont}, this comes from $4(\Delta_{\frac{\gamma}{2}}-1)= -2(2-d)$  where $\Delta_\alpha=\frac{\alpha}{2}(Q-\frac{\alpha}{2})$. This is equivalent to $d=2\Delta_{\frac{\gamma}{2}}$, which is an instance of the Knizhnik-Polyakov-Zamolodchikov (KPZ) relation.
\end{remark}

\section{The scaling limit on random quandragulation decorated by self-avoiding loop}\label{subsec:rpm}
% !TeX spellcheck = en_US
In this section we prove Theorem \ref{thm:saw-loop}. We start by introducing more precisely the objects appearing in the theorem. Recall that a planar map is a connected graph drawn on the sphere $\BB S^2$ such that no two edges cross, viewed modulo an orientation-preserving homeomorphisms from the sphere to itself. A quadrangulation is a planar map such that all faces have four edges. Le Gall and Miermont \cite{miermont-brownian-map,legall-uniqueness} proved that uniformly sampled quadrangulations converge in the scaling limit to the metric measure space known as the Brownian map for the so-called Gromov-Hausdorff-Prokhorov topology \cite{adh-ghp}.

Define the following constants:
\eqb
\lambda = 12,\qquad
\theta = 54,\qquad
\aexp = 5/2\qquad
\bexp = 1/2.
\label{eq:aexp-bexp}
\eqe
 The constants are chosen such that the number of quadrangulations of a $2p$-gon with $m$ faces is of order $\theta^p p^{-\bexp} \lambda^{m}m^{-\aexp}$ for $m\geq cp^2$ for arbitrary fixed $c>0$ \cite{brown-quad-disk-enum}.\footnote{Our quadrangulated $2p$-gons are unrooted. If we consider maps with a root edge on its boundary then the number of maps is of order $\theta^p p^{-b+1} \lambda^{m}m^{-a}$ instead.}
Let $\MS^n$ denote the measure on quadrangulations such that a quadrangulation $M$ with $m$ faces has weight $n^{\aexp} \lambda^{-m}$. 
For $M$ sampled from $\MS^n$, we view $M$ as a metric measure space by considering the graph metric rescaled by $2^{-1/2}n^{-1/4}$ and by giving each vertex mass $2(9n)^{-1}$. With this choice of rescaling, the measure of the set of quadrangulations with mass of order 1 will be of order 1 since the number of quadrangulations with $m$ faces is of order $\lambda^{m}m^{-a-1}$  
\cite{tutte-census}.

If $M$ is a quadrangulation we say that $\eta$ is a \emph{self-avoiding loop} on $M$ if $\eta$ is an ordered set of edges $e_1,\dots,e_{2k} \in \cE(M)$ such $e_j$ and $e_i$ share an end-point if and only if $|i-j|\leq 1$ or $(i,j)\in\{(1,2k),(2k,1) \}$. Let $\#\eta=2k$ denote the number of edges on $\eta$. 
Let $\MS^{n}\otimes \op{SAW}^n$ denote the measure on pairs $(M,\eta)$ where $\eta$ is a self-avoiding loop on $M$ and a pair $(M,\eta)$ has weight 
$$
n^{2a+b-3}\lambda^{-\#\cF(M)}\theta^{-\#\eta}. 
$$
For $(M,\eta)$ sampled from $\MS^{n}\otimes \op{SAW}^n$,  we view $M$ as a metric measure space as above and 
view $\eta$ as a loop on this metric measure space such that the time it takes to trace each edge on the loop is  $2^{-1}n^{-1/2}$. Here we include the edges in the metric-measure structure of $M$ so that $\eta$ can be defined as a continuous curve on $M$; see e.g.\ \cite[Remark 2.4]{gwynne-miller-saw}.

It was proved by Miller and Sheffield that quantum surfaces with $\gamma=\sqrt{8/3}$ can be identified with Brownian surfaces \cite{lqg-tbm1, lqg-tbm2, lqg-tbm3}. More precisely, a quantum surface sampled from $\QS$ with $\gamma=\sqrt{8/3}$ defines a random metric measure space which is equal in law to the Brownian map. 
In particular, a sample from $\QS\otimes \SLE_{8/3}^{\op{loop}}$ with $\gamma=\sqrt{8/3}$ can be viewed as a loop-decorated metric measure space. We will use this interpretation in this subsection and in the statement of Theorem \ref{thm:saw-loop}; this is a slight abuse of notation since we view $\QS\otimes \SLE_{8/3}^{\op{loop}}$ as a measure on the space of loop-decorated LQG surface in other sections. The loop is parametrized by its quantum length.

The paragraphs above allow for a precise statement of Theorem \ref{thm:saw-loop}. We will now turn to the proof of this theorem, which builds on Theorem \ref{thm-loop} along with three ingredients given below: Theorem \ref{thm:gwynne-miller}, Observation \ref{obs:bijection}, and \eqref{eq:rpm-enum}. In order to state these results we first introduce some further notation.  

A planar map $M$ is a \emph{quadrangulated disk} if it is a planar map where all faces have four edges expect for a distinguished face (called the exterior face) which has arbitrary degree and simple boundary. We let $\partial M$ denote the edges on the boundary of the exterior face, and we call $\# \partial M$ the boundary length of $M$. Let $\MD^n$ be the measure on quadrangulated disks such that each quadrangulated disk $M$ has mass $n^{\aexp+\bexp/2-3/2} \lambda^{-\# \cF(M)}\theta^{-(\#\partial M)/2}$.
We need to choose this mass in order for Observation \ref{obs:bijection} below to be correct; note in particular that the exponents of $n$ and $\theta$ have been divided by two as compared to MS$^n$ above since we glue together two disks to form a sphere.
If $M\sim \MD^n$ then we view $M$ as a metric measure space by applying the same rescaling as for $\MS^n$ above. 
For $k\in\N$ let $\MD^n(k)$ denote $\MD^n$ restricted to quadrangulations with boundary length $2k$, and let $\MD^n(k)^\#$ denote $\MD^n(k)$ renormalized to be a probability measure. 

If $M_1,M_2$ are quadrangulated disks with boundary length $2k$ then we can form a quadrangulation with a self-avoiding loop by choosing uniform boundary edges $e_1\in\partial M_1,e_2\in\partial M_2$ and then identifying the boundaries of $M_1,M_2$ such that $e_1$ and $e_2$ are identified. The self-avoiding loop on the sphere represents the boundaries of $M_1,M_2$, and we parametrize the loop so that each edge on the loop has length $2^{-1}n^{-1/2}$. Note that the scaling we use of distances along the loop ($2^{-1}n^{-1/2}$) is different from the scaling we use of graph distances in the map ($2^{-1/2}n^{-1/4}$); this choice of exponents ($-1/2$ and $-1/4$) cause both distances to be asymptotically non-trivial. If $M_1,M_2\sim \MD^n(k)^\#$ then we denote the measure on spheres decorated with a self-avoiding loop sampled in this way by $\op{Weld}(\MD^n(k)^\#, \MD^n(k)^\#)$. 
\begin{theorem}[\cite{gwynne-miller-simple-quad,gwynne-miller-gluing}]
	For any $\ell>0$ the following convergence in law holds for the Gromov-Hausdorff-Prokhorov-uniform topology 
	$$
	\op{Weld}(\MD^n(\lceil\ell n^{1/2}\rceil )^\# , \MD^n(\lceil\ell n^{1/2}\rceil)^\# ) 
	\Rightarrow \op{Weld}(\QD(\ell)^\# , \QD(\ell)^\# ).
	$$
	\label{thm:gwynne-miller}
\end{theorem} 
\begin{proof}
	\cite[Theorem 1.5]{gwynne-miller-simple-quad} proves this convergence result when the right side is given by a metric space quotient. By \cite{gwynne-miller-gluing} and local absolute continuity we get that this metric space quotient gives the same metric space as the conformal welding of the two disks.
\end{proof}

Let $Z_n(k)$ denote the total mass of $\MD^n(k)$. It follows from \cite{brown-quad-disk-enum} (see his enumeration result cited right below \eqref{eq:aexp-bexp}) that there is a constant $C>0$ such that
\eqb
	\frac{Z_n(k)}{n^{a+b/2-3/2}k^{-b-2a+2}}=C(1+o_k(1)),
\label{eq:rpm-enum}
\eqe
where the $o_k(1)$ is uniform in $n$. We now define $\op{Weld}(\MD^n, \MD^n )$ in the same spirit as $\op{Weld}(\QD,\QD)$
\eqb
\begin{split}
	&\op{Weld}(\MD^n, \MD^n )
	:=  
	\sum_{k=1}^{\infty}  2kZ_n(k)^2
	\op{Weld}(\MD^n(k)^\#, \MD^n(k)^\# ),
\end{split}
\label{eq:discrete-weld}
\eqe 
where we recall that samples from  $\MD^n(k)$ have boundary length $2k$ and  $\MD^n(k)=Z_n(k) \MD^n(k)^\#$.

The observation we state next is immediate by combinatorial considerations and was also observed in slightly different forms in e.g.\ \cite[Section 1.3.3]{gwynne-miller-simple-quad} 
and 
\cite{caraceni-curien-saw}. The key point is that there are $2k$ ways of welding together two samples from $\MD^n(k)$.
\begin{observation}
	$\op{Weld}(\MD^n, \MD^n )=\MS^{n}\otimes \op{SAW}^n$.
	\label{obs:bijection}
\end{observation}

Combining the three ingredients above, we can now conclude the proof of Theorem \ref{thm:saw-loop}.
\begin{proof}[Proof of Theorem \ref{thm:saw-loop}]
	For constants $c_1,c_2>0$,
	\eqb
	\begin{split}
		\op{Weld}(\MD^n, \MD^n )|_{A(c)}
		&=  
		\sum_{k=\lceil c\sqrt{n}\rceil}^{\lceil c^{-1}\sqrt{n}\rceil}  2kZ_n(k)^2 
		\op{Weld}(\MD^n(k)^\#, \MD^n(k)^\# )\\
		&\Rightarrow c_1 \int_c^{c^{-1}} \op{Weld}(\QD(\ell)^\# , \QD(\ell)^\# ) \ell^{-2(b+2a-2)+1} \,d\ell\\
		&= c_2 \int_c^{c^{-1}} \op{Weld}(\QD(\ell) , \QD(\ell) ) \ell \,d\ell,
	\end{split}
	\label{eq:disk-gluing-convergence}
	\eqe
	where we use in the last step that the total mass of $\QD(\ell)$ is a power law with exponent $-7/2 =-(b+2a-2)$, which follows e.g.\ from Lemma \ref{lem:QD-len}. The right side of \eqref{eq:disk-gluing-convergence} is equal to $c_0\cdot \QS\otimes \SLE_{8/3}^{\op{loop}}|_{A(c)}$ for some $c_0>0$ by Theorem \ref{thm-loop}, while it follows from Observation \ref{obs:bijection} that the left side of  \eqref{eq:disk-gluing-convergence} is equal to $\MS^{n}\otimes \op{SAW}^n|_{A(c)}$. This concludes the proof.
\end{proof}

\bibliographystyle{hmralphaabbrv}
\bibliography{cibib}

\end{document}